\documentclass{amsart}
\usepackage[utf8]{inputenc}

\author{Ahmed Laghribi}
\address{Univ. Artois, UR2462, Laboratoire de Mathématiques de Lens (LML), F-62300 Lens, France} 
\email{ahmed.laghrbi@univ-artois.fr}

\title[Cohomological invariants in characteristic 2]{A cohomological invariant for algebras of degree 8 and exponent 2 in characteristic 2}

\author{Nico Lorenz}
\address{{Fakult\"at f\"ur Mathematik}, {Ruhr-Universit\"at Bochum},{Universit\"atsstra{\ss}e 150}, {44780 Bochum}, {North Rhine-Westphalia}, {Germany}}
\email{nico.lorenz@ruhr-uni-bochum.de}

\date{\today}

\usepackage{amsthm}
\usepackage{amsmath}
\usepackage{amssymb}
\usepackage[shortlabels]{enumitem}
    \setlist[enumerate,1]{label=(\roman*), font=\normalfont}
    \setlist[enumerate,2]{label=(\alph*), font=\normalfont}
    \setlist{nosep}
\usepackage{tikz}
\usepackage{xcolor}
\usepackage{mydiagrams}
\usepackage{todonotes}

\definecolor{Sepia}{HTML}{671800}
\definecolor{MidnightBlue}{HTML}{006795}
\usepackage[colorlinks=true]{hyperref}
\hypersetup{citecolor=Sepia, linkcolor=Sepia, urlcolor=MidnightBlue}

\usepackage{cleveref}

\newtheorem{definition}{Definition}[section]

\newtheorem{theorem}[definition]{Theorem}
\newtheorem{lemma}[definition]{Lemma}
\newtheorem*{lemma*}{Lemma}
\newtheorem{corollary}[definition]{Corollary}

\newtheorem{remark}[definition]{Remark}
\newtheorem{example}[definition]{Example}
\newtheorem{proposition}[definition]{Proposition}

\numberwithin{equation}{section}

\newcommand{\Z}{\mathbb{Z}}
\newcommand{\N}{\mathbb{N}}

\newcommand{\Q}{\mathbb{Q}}

\DeclareMathOperator{\cha}{char }

\newcommand{\im}{\operatorname{im}}

\newcommand{\qf}[1]{\langle #1\rangle}

\newcommand{\bb}{\mathfrak{b}} 
\newcommand{\iw}[1]{\mathfrak{i}_{\mathrm{W}}(#1)} 
\newcommand{\an}{\mathrm{an}} 
\newcommand{\ql}[1]{\mathrm{ql}\left(#1\right)}

\newcommand{\qpf}[1]{\langle\!\langle #1]\!]}
\newcommand{\bif}[1]{\langle #1 \rangle_b}
\newcommand{\bipf}[1]{\langle\!\langle #1\rangle\!\rangle_b}
\makeatletter

\newcommand{\bigperp}{%
  \mathop{\mathpalette\bigp@rp\relax}%
  \displaylimits
}
\newcommand{\bigp@rp}[2]{%
  \vcenter{
    \m@th\hbox{\scalebox{\ifx#1\displaystyle2.1\else1.5\fi}{$#1\perp$}}
  }%
}
\makeatother

\newcommand{\dbrac}[1]{(\!(#1)\!)}
\newcommand{\dbrack}[1]{[\![#1]\!]}

\renewcommand{\H}{\mathbb{H}}

\newcommand{\BrTwo}{\mathrm{Br}_2}
\newcommand{\Br}{\mathrm{Br}}
\newcommand{\ind}{\mathrm{ind}}
\newcommand{\quat}[1]{[#1)}
\newcommand{\trace}{\mathrm{Tr}}
\newcommand{\norm}{\mathrm{N}}
\newcommand{\res}{\mathrm{res}}

\newcommand{\diff}{\mathrm{d}}
\newcommand{\dlog}{\mathrm{dlog}}

\newcommand{\I}{\mathsf{I}}
\newcommand{\W}{\mathsf{W}}
\newcommand{\GP}{\mathsf{GP}}

\newcommand{\D}{\mathsf{D}}
\newcommand{\inv}{\mathsf{inv}}
\newcommand{\torsChow}{\mathrm{Tors CH}^2}
\newcommand{\SevBrauer}{\mathrm{SB}}
\renewcommand{\phi}{\varphi}


\begin{document}
\begin{abstract}
    Our aim in this paper is to extend a work of Sivatski \cite{Sivatski_CohoInv} to characteristic 2.
	More precisely, for $F$ a field of characteristic $2$ and a central simple algebra $A$ of exponent 2 that splits over a triquadratic extension of $F$ of separability degree at least 4, we attach a cohomological invariant $\inv(A) \in H_2^3(F) / G$, where $H_2^3(F)$ is the third Kato-Milno cohomology group and $G$ is a subgroup of $H_2^3(F)$ divisible by the Brauer class of $A$.
	As an application, we will relate the decomposability of the algebra in degree 8 to the vanishing of $\inv(A)$.
	Moreover, we will use this invariant to prove some descent results for central simple algebras and quadratic forms over biquadratic extensions.
\end{abstract}

\maketitle

\section{Introduction} \label{Sec_Intro}

Throughout this article, $F$ denotes a field of characteristic 2. 
Given a central simple algebra $A$ over $F$, it is a difficult problem to know whether it is decomposable or not, i.e. whether there are nontrivial central simple algebras $A_1, A_2$ over $F$ such that $A \cong A_1 \otimes_F A_2$.
Recall that for $i\in \{1, 2\}$, every central simple algebra of degree $2^i$ and exponent $2$ is Brauer equivalent to a product of $i$ quaternion algebras.
The situation is more complicated for degree at least $8$ and exponent $2$.
In this case, Amitsur-Rowen-Tignol \cite[Theorem 5.1]{AmitsurRownTignol_DivisionAlgebras} in characteristic different from 2 and Rowen \cite[Section 3]{Rowen_DivisionAlgebraExp2Char2} in characteristic 2 provided examples of indecomposable algebras of degree 8 and exponent 2.
Later, Karpenko produced examples of indecomposable algebras of degree $p^n$ and exponent $p$ except for the case $p = 2 = n$, see \cite[Corollary 5.4]{Karpenko_Codim2Cycles}.

Another important topic of a somewhat different flavour is the descent problem: given a field extension $K / F$ and a central simple algebra $A$ over $K$, determine whether there is an $F$-algebra $A'$ such that $A \cong A'_K$.
Recently, in the spirit of what was done before by Barry in \cite{Barry_DecompAlgebrasDeg8Exp2}, Barry, Chapman and Laghribi \cite{BarryChapmanLaghribi_DescentBiquaternion} studied the descent problem over separable quadratic extensions $K/F$ for biquaternion algebras with trivial corestriction.
They introduced an invariant $\delta_{K/F}$ that vanishes for such a $K$-algebra $B$ if and only if $B$ is defined over $F$ up to isomorphism.
Moreover they applied this invariant to construct an indecomposable algebra of degree $8$ and exponent $2$ over a field of cohomological $2$-dimension 3, demonstrating a connection between decomposability and descent.

In the spirit of finding invariants that encode information on the algebra, Sivatski introduced in \cite{Sivatski_CohoInv}, for fields of characterstic not 2, an invariant for any central simple algebra of exponent $2$ split by a triquadratic extension.
In particular, when the degree is $8$, this invariant vanishes if and only if the algebra decomposes into a tensor product of three quaternion algebras related to the triquadratic extension.
It has been also applied to prove results about the descent for algebras and quadratic forms. 

Our aim in this paper is to extend most of \cite{Sivatski_CohoInv} to fields of characteristic 2.
We will not restrict ourselves to separable triquadratic extensions but will consider more generally the the case of extensions of the form $L = K(\wp^{-1}(a), \wp^{-1}(b))$, where $K/F$ is an arbitrary quadratic extension.
In the next section, we will recall some background on quadratic forms and differential forms in characteristic 2.
After that, in \Cref{sec_ConstructionInvariant} we construct the invariant 
\[\inv(A) \in H_2^3(F) / [A] \wedge \dlog(\im(\norm_{K(\wp^{-1}(a + b)) / F})\setminus \{0\}),\]
where $H_2^3(F)$ is the third Kato-Milne cohomology group, $[A]$ is the class of $A$ in $\BrTwo(F)$ and $\dlog(\alpha) = \frac{\diff \alpha}{\alpha} \in \Omega^1_F$ for any $\alpha \in F^\ast$ with $\Omega^1_F$ the space of 1-differential forms, see \Cref{Sec_Preliminaries} for details.

In \Cref{sec_InvTrivial}, we prove that this invariant $\inv(A)$ is trivial if and only if $[A] = [Q \otimes \quat{a+b, v} \otimes Q']$, where $Q, Q'$ are quaternion such that $Q'$ splits over $K$.
These conditions are further equivalent to $[A_{K(\wp^{-1}(a + b))}]$ being defined over $F$, see \Cref{prop_invTrivial}.
In \Cref{ex_NontrivialInv}, we construct a biquaternion $F$-algebra $B$ such that $\inv(B) \neq 0$ but $\delta_{K/F} (B) = 0$.
This examples combines involved arguments relying on the construction of Rowen of an indecomposable algebra of degree 8 and exponent 2, residue maps introduced in \cite[Proposition 6.3]{BarryChapmanLaghribi_DescentBiquaternion} and the nonexcellence of biquadratic separable extensions \cite[Example 6.1]{LaghribiMukhija_ExcellenceInsepQuartic}, which is also based on Rowen's example cited before.

In \Cref{sec_ChowGroups} we will combine the results about the invariant with deep results from E. Peyre and N. Karpenko to construct an indecomposable algebra of index $8$ such that there is a nontrivial element in the torsion part of the second Chow group of its Severi-Brauer variety. 
Moreover, we can describe this element explicitly.

\Cref{sec_Descent} will be devoted to applications of the invariant. 
More precisely for an odd degree extension $M / F$, mixed biquadratic extension $M' / F$ and a central simple $F$-algebra $D$ such that $D_{MM'}$ is Brauer equivalent to $Q_{MM'}$ for an $M$-quaternion $Q$, we show that $D_{M'}$ comes from an $F$-quaternion algebra, see \Cref{thm_DescentAlgebras}.

We apply this descent to get a similar result for quadratic forms.
For an $F$-quadratic form $\varphi$ such that $\varphi_{MM'}$ is Witt equivalent to $\psi_{MM'}$ for an $M$-quadratic form $\psi$ of dimension at most 3 or similar to a two-fold Pfister form, we show that $\varphi_{M'}$ is Witt equivalent to some $F$-quadratic form $\psi'$ of the same shape as $\psi$ over $M'$, see \Cref{thm_descentQF}.

%
%
%
%
%
%
%
%
%

\section{Preliminaries and Notation}\label{Sec_Preliminaries}

\subsection{Quadratic forms}

For quadratic forms, our standard reference will be \cite{ElmanKarpenkoMerkurjev2008}.
To keep our paper self contained, we recall some definitions and notations.
Over a field of characteristic 2, any quadratic form $\varphi$ decomposes as
\begin{align} \label{eq_QFDecomp}
	\varphi \cong [a_1, b_1] \perp \ldots \perp [a_r, b_r] \perp \qf{c_1, \ldots, c_s},
\end{align}
where $[a, b]$ denotes the two-dimensional quadratic form given by $aX^2 + XY + bY^2$ and $\qf{c_1, \ldots, c_s}$ denotes the $s$-dimensional quadratic form given by $c_1Z_1^2 + \ldots + c_sZ_s^2$.
The tuple $(r, s) \in \N \times \N$ is unique and called the \emph{type} of $\varphi$.
The form $\qf{c_1, \ldots, c_s}$ is unique up to isometry and called the \emph{quasilinear part} of $\varphi$.
We denote it by $\ql{\varphi}$.
The form $\varphi$ is called \emph{nonsingular} if $s = 0$ and \emph{totally singular} if $r = 0$.
In this case, the \emph{Arf-invariant} of $\varphi$ as in \eqref{eq_QFDecomp} is defined to be $\Delta(\varphi) := \sum_{i = 1}^r a_ib_i \in F/ \wp(F)$, where $\wp(F) = \{x^2 + x \mid x \in F\}$.

A quadratic form $\varphi$ is called \emph{isotropic} if there is a nonzero vector $v$ with $\varphi(v) = 0$; otherwise, it is called \emph{anisotropic}.
The quadratic form $\H = [0, 0]$ is called the \emph{hyperbolic plane}. 
Up to isometry it is the unique nonsingular isotropic form of dimension 2.
A form isometric to a sum of hyperbolic planes is called \emph{hyperbolic}.

The Witt decomposition asserts that any form $\varphi$ uniquely decomposes as
\[\varphi \cong \varphi_\an \perp \mathcal H \perp \qf{0, \ldots, 0}, \]
where $\varphi_\an$ is anisotropic and $\mathcal H$ is hyperbolic.
The integer $\frac12 \dim \mathcal H$ is called the \emph{Witt index} of $\varphi$ and denoted by $\iw\varphi$.
We call $\varphi_\an$ the \emph{anisotropic part} of $\varphi$.

Two forms $\varphi, \psi$ are \emph{Witt equivalent}, if there exist $n, m \in \N$ such that 
\[\varphi \perp n \times \H \cong \psi \perp m \times \H.\]
In this case, we write $\varphi \sim \psi$.

We denote the \emph{(quadratic) Witt group} of $F$ by $\W_q(F)$, and the \emph{(bilinear) Witt ring} of $F$ by $\W(F)$.
It is well known that we have an action of $\W(F)$ on $\W_q(F)$ induced by the tensor product, see \cite[Section 8.B]{ElmanKarpenkoMerkurjev2008}.
For any $n \geq 1$, let $\I^n(F)$ denote the $n$-th power of the fundamental ideal $\I(F)$ of $\W(F)$ (where $\I^0(F) = \W(F)$).
Similarly, denote by $\I^{n+1}_q(F)$ the submodule of $\W_q(F)$ given by $\I^n(F) \otimes \W_q(F)$ (note that $\I^1_q(F) = \W_q(F)$).
It is well known that $\I^n(F)$ is additively generated by the bilinear forms
\[\bif{1, a_1} \otimes \ldots \otimes \bif{1, a_n} =: \bipf{a_1, \ldots, a_n}\]
for some $a_1, \ldots, a_n \in F^\ast$ that we call an \emph{$n$-fold bilinear Pfister form}.
Similarly, $\I^{n+1}_q(F)$ is additively generated by scalar multiples of forms of the shape
\[\bipf{a_1, \ldots, a_n} \otimes [1, a] =: \qpf{a_1, \ldots, a_n, a}\]
which are called \emph{$(n+1)$-fold quadratic Pfister forms}.

Recall the Arason-Pfister Hauptsatz, which states that if $\varphi$ is a nonzero anisotropic quadratic Pfister form lying in $\I^n(F)$, then $\dim(\varphi) \geq 2^n$.

Given a field extension $K / F$ of finite degree and a nonzero $F$-linear map $s: K \to F$ and a (quadratic or bilinear) form $\varphi$ over $K$, we associate the form, denoted by $s_\ast(\varphi) := s \circ \varphi$, that we call the \emph{transfer} of $\varphi$ with respect to $s$.
This new form has dimension $[K : F] \cdot \dim(\varphi)$, where $\dim(\varphi)$ denotes the dimension of $\varphi$.
A well known fact that we will regularly use is the \emph{Frobenius reciprocity}, stating that for a bilinear form $\beta$ over $F$ resp. over $K$ and a quadratic form $\varphi$ over $K$ resp. over $F$, we have
\[s_\ast(\beta_K\otimes \varphi) = \beta \otimes s_\ast(\varphi) \quad \text{ resp. } \quad s_\ast(\beta\otimes \varphi_K) = s_\ast(\beta) \otimes \varphi.\]

\subsection{Central simple algebras}

Our standard reference for central simple algebras is \cite{GilleSzamuely_CSAGalCohomology}.
Let $\Br(F)$ denote the Brauer group of $F$.
For any central simple algebra $A$ over $F$, let $\deg(A) = \sqrt{\dim_F(A)}$ denote the \emph{degree} of $A$, and let $\ind(A)$ be the \emph{index} of $A$, i.e. the degree of the division algebra $D$ Brauer equivalent to $A$.
An algebra $A$ is called \emph{split} if it is Brauer equivalent to $F$.

An \emph{$F$-quaternion algebra}, denoted by $\quat{a, b}$, is an $F$-algebra of dimension $4$ with a basis $1, u, v, w$ subject to the relations 
\[u^2 + u = a\in F, \quad v^2 = b \in F^\ast, \quad uv = w = vu + v.\]
For any quaternion algebra $Q = \quat{a, b}$, we can associate its \emph{norm form} given by $\qpf{b, a}$.
Recall the following relations in $\BrTwo(F)$ that will be used frequently:
\[\quat{a, b} + \quat{a, c} = \quat{a, bc} \quad \text{ and }\quad \quat{x, y} + \quat{z, y} = \quat{x + z, y},\]
where $a, x, z \in F$ and $b, c, y \in F^\ast$.

Note that $Q$ contains an inseparable and a separable quadratic field extensions of $F$, namely $F(v)$ and $F(u)$ respectively and $Q$ splits over these two extensions.
Conversely, if a central simple algebra $A$ of exponent 2 splits over a quadratic extension $F(x)$ given by $x^2 = d \in F$ (resp. $x^2 + x = d \in F$) then $A$ is Brauer equivalent to $\quat{c, d}$ (resp. $Q \cong \quat{d, c}$) for a suitable $c \in F$ (resp. $c \in F^\ast$), i.e. we have
\[\BrTwo(F(x) / F) := \ker(\BrTwo(F) \to \BrTwo(F(x))) = \begin{cases}
	\{[\quat{c, d}] \mid c \in F\}, &\text{ if $x^2 = d$}\\
	\{[\quat{d, c}] \mid c \in F^\ast\}, &\text{ if $x^2 + x = d$}.
\end{cases}
\]

Every central simple algebra of degree 4 and exponent 2 is a biquaternion algebra, i.e. it is isomorphic to the tensor product of two quaternion algebras (see \cite[p. 369]{Albert_Biquaternion} or \cite[(16.1) Theorem]{KnusMerkujevRostTignol_Involutions} for a modern treatment).
    An \emph{Albert form} is a form of dimension $6$ with trivial Arf invariant. 
    It is thus of the shape $x_1[1, a] \perp x_2[1, b] \perp x_3[1, a+b]$ for some $x_1, x_2, x_3, a, b \in F$.
    
    Given a biquaternion algebra $A = \quat{a_1, b_1} \otimes \quat{a_2, b_2}$, we can attach an Albert form $\alpha = [1, a_1 + a_2] \perp b_1[1, a_1] \perp b_2[1, a_2]$.
    This form is unique up to similarity, see \cite{MammoneShapiro_AlbertForm}.
    The isotropy behaviour of the $\alpha$ encodes the index of $A$ in that $\ind(A) = 4$ if and only if $\iw\alpha = 0$, $\ind(A) = 2$ if and only if $\iw\alpha = 1$ and $\ind(A) = 1$ if and only if $\iw\alpha = 3$, see \cite[(16.5) Theorem]{KnusMerkujevRostTignol_Involutions}.
   
Let $C(\varphi)$ denote the Clifford algebra of the quadratic form $\varphi$.
If $\varphi$ is nonsingular, then $C(\varphi)$ is a central simple algebra.
Explicitly, for $\varphi \cong b_1[1, a_1] \perp \ldots \perp b_r[1, a_r]$ with $b_1, \ldots, b_r \neq 0$, we have
\[C(\varphi) = \bigotimes_{i=1}^r \quat{a_i, b_i}\]
by \cite[p. 342]{ScharlauQfHFBook}.
The \emph{Clifford invariant} $c(\varphi)$ of $\varphi$ is defined as the class of $C(\varphi)$ in $\BrTwo(F)$, where $\BrTwo(F)$ denotes the 2-torsion part of the Brauer group of $F$.

\subsection{Differential forms}

Let $\Omega^1_F$ denote the $F$ vector space of 1-dimensional differentials generated by symbols $\diff a$, for $a \in F$, subject to the relations
\[\diff(a + b) = \diff(a) + \diff(b) \quad \text{ and } \quad \diff(ab) = a\diff b + b\diff a.\]
For an integer $n \geq 1$, let $\Omega_F^n = \bigwedge_{i = 1}^n \Omega_F^1$.
We take $\Omega_F^0 = F$ and for $n < 0$, we take $\Omega_F^n = 0$.
The space $\Omega_F^n$ is generated as an $F$ vector space by symbols of the shape $a \diff b_1 \wedge \ldots \wedge b_n$.

Let $\diff: \Omega_F^{n} \to \Omega_F^{n+1}$ be the \emph{differential operator} defined by 
\[a \diff b_1 \wedge \ldots \wedge b_n \mapsto \diff a \wedge \diff b_1 \wedge \ldots \wedge b_n.\]

The \emph{Artin-Schreier operator} $\wp: \Omega_F^n \to \Omega_F^{n} / \diff \Omega_F^{n - 1}$ is defined on generators by 
\[a \frac{\diff b_1}{b_1} \wedge \ldots \wedge \frac{\diff b_n}{b_n} \mapsto (a^2 - a)\frac{\diff b_1}{b_1} \wedge \ldots \wedge \frac{\diff b_n}{b_n}.\]

Its cokernel is denoted by $H_2^{n+1}(F)$.
Explicitly we have
\[H_2^{n+1}(F) = \Omega_F^n / (\diff \Omega_F^{n-1} + \wp(\Omega_F^n)).\]

It is easy to see that $H_2^1(F) \cong F / \wp(F)$ and a well known result by Sah \cite[Theorem 2 (ii)]{Sah_SBFandQF} asserts that $H_2^2(F) \cong \BrTwo(F) \cong \I_q^2(F) / \I_q^3(F)$.
More generally, for each $n \geq 1$, we have a celebrated result by Kato \cite{Kat1} stating that we have isomorphisms
\[f^n: \I_q^n(F) / \I_q^{n+1}(F) \to H_2^n(F) \]
mapping $\overline{\qpf{b_1, \ldots, b_{n-1}, a}}$ to $\overline{a \frac{\diff b_1}{b_1} \wedge \ldots \wedge \frac{\diff b_n}{b_n}}$.
For any $\varphi \in \I^n(F)$, let $e^n(\varphi) = f^n(\overline{\varphi})$.
To slim down the notation, for a central simple algebra $A$ over $F$ with $[A] \in \BrTwo(F)$, we also use the notation $[A]$ for the corresponding element in $H_2^2(F)$.

One of the results we will repeatedly use concerns the kernel $H_2^{n+1}(K/F) := \ker(H_2^{n+1}(F) \to H_2^{n+1}(K))$ of the restriction map, where $K / F$ is a biquadratic extension.
More precisely, if $K_1, K_2$ are quadratic extensions of $F$ with $K_1K_2 = K$, this kernel is equal to
\[H_2^{n+1}(K_1 / F) + H_2^{n+1}(K_2 / F),\]
see \cite[Theorem 20]{AravireJacob_GradedKernelBiquadraticSeparable} for the separable case and \cite[Proposition 2 and 3]{AravireLaghribiMultiQuadratic} for all other cases.

For $n = 1$, since $\BrTwo(F) \cong H_2^2(F)$, we obtain
\begin{align} \label{BrauerKernelBiquadratic}
	\BrTwo(K/F) = \BrTwo(K_1/F) + \BrTwo(K_2/F).
\end{align}

\section{Construction of the invariant} \label{sec_ConstructionInvariant}

Throughout this section, let $K / F$ be a quadratic field extension.
Then $K = F(\delta)$, where either $\delta^2 + \delta = d \in F^\ast$ or $\delta^2 = d \in F^\ast$, i.e. $\delta = \wp^{-1}(d)$ or $\delta = \sqrt d$, depending on whether $K / F$ is separable or inseparable respectively.
We fix the notation for $d$ and $\delta$ throughout the text.
Let $s: K \to F$ be the $F$-linear map with $s(1) = 0$ and $s(\delta) = 1$.
In the case of a separable extension, $s$ thus coincides with the usual trace map.
We denote by $s_\ast$ the induced Scharlau transfer on the level of Witt groups $\W_q(K) \to \W_q(F)$ and on the level of Kato-Milne cohomology $H_2^{m+1}(K) \to H_2^{m+1}(F)$ (and thus also on the level of Brauer classes in the 2-torsion part).
Let $\norm_{K/F}$ denote the norm map corresponding to the extension $K / F$, which is given by:
\[\norm_{K/F}(x) = \begin{cases}
x\overline x & \text{ if $K/F$ is separable}\\
x^2 & \text{ if $K/F$ is inseparable,}
\end{cases}\]
where $\overline x$ denotes the conjugate of $x$ in the separable case.

We will make extensive use of the following well known result:
\begin{lemma} \label{lem_basicComputationsTransfer}
	Let $K / F$, $s: K \to F$ and $\norm_{K/F}$ be as in the preceding paragraph.
	
	We have the following:
	\begin{enumerate}
		\item \label{lem_basicComputationsTransfer1} for $x\in K$, we have $s_\ast(\bif x) = s(x)\bipf{\norm_{K/F}(x)}$.
		\item \label{lem_basicComputationsTransfer2} for $a \in F$ and $x \in K^\ast$, we have $s_\ast(\quat{a, x}) = \quat{a, \norm_{K/F}(x)}$.
	\end{enumerate}
\end{lemma}
\begin{proof}
	\ref{lem_basicComputationsTransfer1} is elementary; for \ref{lem_basicComputationsTransfer2} see \cite[2.5 Lemma]{ArasonAravireBaeza_Invariants} resp. \cite[Proposition 3.2]{Bingol_SymbolLength}.
\end{proof}

Let now further $a, b \in F^\ast$ such that for $L = K(\wp^{-1}(a), \wp^{-1}(b))$ we have $[L:F] = 8$.
We consider a central simple algebra $D$ with $[D] \in \BrTwo(L/F)$.
We further put 
\[E = K(\wp^{-1}(a + b))\] and
\[N(d; a, b) := \overline{\dlog(\im(\norm_{E/F})\setminus\{0\}),}\]
where $\dlog(\alpha) = \frac{\diff \alpha}{\alpha}$ for $\alpha \in F^\ast$.

To keep a picture of all the fields in place we include the following diagram of field extensions:
\begin{align*}
	\begin{diagram}
		& & L = K(\wp^{-1}(a), \wp^{-1}(b)) & & \\
		& \ruLine & \dLine & \rdLine & \\
		K(\wp^{-1}(a)) & & E = K(\wp^{-1}(a+b)) & & K(\wp^{-1}(b))\\
		& \rdLine & \dLine & \ruLine & \\
		& & K & & \\
		& & \dLine & & \\
		& & F & & 
	\end{diagram}
\end{align*}

For the case when $K / F$ is inseparable, we need the following analogue of \cite[Theorem 34.9]{ElmanKarpenkoMerkurjev2008}.

\begin{lemma}\label{lem_TransferExactSequenceInsep}
    Let $K / F$ be an inseparable quadratic extension, $K = F(\sqrt d)$, $s: K \to F$ be the linear map with $s(1) = 0, s(\sqrt d) = 1$ and $s_\ast:\W_q(K) \to \W_q(F)$ be the induced Scharlau transfer. 
    Then 
    \[\begin{diagram}
        \W_q(F) & \rTo^{r_\ast} & \W_q(K) & \rTo^{s_\ast} & \W_q(F)
    \end{diagram}\] 
    is exact.
\end{lemma}
\begin{proof}
    It is known that $s_\ast \circ r_\ast = 0$ by \cite[Lemma 34.14]{ElmanKarpenkoMerkurjev2008}.
    So let now $\varphi$ be an anisotropic nonsingular form over $K$ such that $s_\ast(\varphi)$ is hyperbolic and let $V$ be the underlying vector space.
    Since $s_\ast(\varphi)$ is hyperbolic, it is isotropic, so there is $v \in V$ with $s(\varphi(v)) = 0$, i.e. $\varphi(v) \in F$.
    Since $\varphi$ is anisotropic, we have $a := \varphi(v) \neq 0$ and we can write $\varphi \cong a[1, x] \perp \varphi'$ for some $x \in K$.
    But $[1, x] \cong [1, x^2]$ and since $K = F(\sqrt d)$, we have $x^2 \in F$.
    Now we compute
    \[0 = s_\ast(\varphi) = s_\ast(a[1, x^2] + \varphi) = s_\ast(a[1, x^2]) + s_\ast(\varphi') = s_\ast(\varphi')\]
    and we conclude by induction on $\dim(\varphi)$.
\end{proof}

We will now give the construction of our invariant.

\begin{proposition} \label{prop_DK_ExistenceInv}
    Let $[D] \in \BrTwo(F)$ be such that $[D_L] = 0$.

    Then, there are $x, y \in K^\ast$ such that
    \[[D_K] = [\quat{a, x} \otimes \quat{b, y}]\]
    and 
    \begin{align*}
        \inv(D) &:= e^3(s_\ast(\qpf{x, a} \perp \qpf{y, b})) \in H_2^3(F) / (N(d; a, b) \wedge [D]),
    \end{align*}
    does not depend on the choice of $x, y$.
\end{proposition}
\begin{proof}
    We have 
    \[\ind(D_K)\mid \ind(D_L)\cdot [L:K] = 4\]
    by \cite[Corollary 4.5.11]{GilleSzamuely_CSAGalCohomology}, so we know that $[D_K] = [B]$ for some biquaternion algebra $B$ defined over $K$.
    Since $[D_K] \in \BrTwo(L/K)$, using \eqref{BrauerKernelBiquadratic}, we get
   \begin{align} \label{eq_D_KAlbertForm}
        [D_K] = [\quat{a, x} \otimes \quat{b, y}]
    \end{align}
    for some $x, y \in K^\ast$.

    By \cite[(6.4) Proposition]{MR1154404} ($K / F$ separable) resp. \cite[Theorem 4.2]{AravireLaghribiORyan_TransferInsep} ($K / F$ purely inseparable) we have an exact sequence
    \begin{align}\label{eq_ComplexExtTransfer}
        \begin{diagram}
            \BrTwo(F) & \rTo^{\res_{K/F}} & \BrTwo(K) & \rTo^{s_\ast} & \BrTwo(F)
        \end{diagram}
    \end{align}
    induced by scalar extension and the transfer induced by $s$ on central simple algebras.
    Using \Cref{lem_basicComputationsTransfer} \ref{lem_basicComputationsTransfer2}, we thus have
    \begin{align*}
        0 = [s_\ast(D_K)] &= [s_\ast(\quat{a, x} \otimes \quat{b, y})]\\
        &= [s_\ast(\quat{a, x})] \cdot [s_\ast(\quat{b, y})]\\
        &= [\quat{a, \norm_{K/F}(x)}] \cdot [\quat{b, \norm_{K/F}(y)}].
    \end{align*}
    This implies $\quat{a, \norm_{K/F}(x)} \cong \quat{b, \norm_{K/F}(y)}$, so for the norm forms of the respective quaternion algebras, we obtain
    \begin{align}\label{eq_EqualityOfPfisterforms}
        \qpf{\norm_{K/F}(x), a} \cong \qpf{\norm_{K/F}(y), b}.
    \end{align}

    We now consider the Albert form $\varphi$ of the biquaternion algebra $D_K$ according to the decomposition in \Cref{eq_D_KAlbertForm}, i.e. $\varphi = [1, a+b] \perp x[1, a] \perp y[1, b]$.
    Using the exactness of 
    \[\begin{diagram}
        \W_q(F) & \rTo^{r_\ast} & \W_q(K) & \rTo^{s_\ast} & \W_q(F)
    \end{diagram}\]
    (see \cite[Theorem 34.9]{ElmanKarpenkoMerkurjev2008} and \Cref{lem_TransferExactSequenceInsep}), Frobenius reciprocity and \Cref{lem_basicComputationsTransfer} \ref{lem_basicComputationsTransfer1}, we obtain
    \begin{align*}
        s_\ast(\varphi) &\cong s_\ast(x[1, a] \perp y[1, b]) \\
        &\cong s_\ast(\bif{x}) \otimes [1, a] \perp s_\ast(\bif y) \otimes [1, b]\\
        &\cong s(x)\bipf{\norm_{K/F}(x)} \otimes [1, a] \perp s(y) \bipf{\norm_{K/F}(y)} \otimes [1, b] \\
        &\cong s(x) \qpf{\norm_{K/F}(x), a} \perp s(y) \qpf{\norm_{K/F}(y), b}.
    \end{align*}
    Using \Cref{eq_EqualityOfPfisterforms}, we finally see
    \begin{align}\label{eq_TrVarphiInI3}
        s_\ast(\varphi) &\cong s(x) \qpf{\norm_{K/F}(x), a} \perp s(y) \qpf{\norm_{K/F}(x), a} \nonumber \\
        &\cong s(x) \qpf{s(x)s(y), \norm_{K/F}(x), a} \in \GP_3(F).
    \end{align}

    Now suppose we have another decomposition 
    \[[D_K] = [\quat{a, x'}] \otimes [\quat{b, y'}]\]
    with corresponding Albert form $\varphi' = [1, a+b] \perp x'[1, a] \perp y'[1, b]$.
    We thus have $\quat{a, xx'} \cong \quat{b, yy'} \in \BrTwo(K(\wp^{-1}(a))/K) \cap \BrTwo(K(\wp^{-1}(b) / K)$, from which we obtain
    \begin{align*}
        \qpf{xx', a} \cong \qpf{yy', b}.
    \end{align*}
    Over $E' = K(\wp^{-1}(b))$, the form 
    \begin{align*}
        [1, a]_{E'} \perp xx'[1, a]_{E'} \cong \qpf{xx', a}_{E'} \cong \qpf{yy', b}_{E'}
    \end{align*}
    is clearly hyperbolic, so that we have $[1, a]_{E'} \cong xx'[1, a]_{E'}$ and by invoking \cite[Lemma 5.3]{ChapmanLaghribiMukhija_I3H3SymbolLength}, we obtain 
    \[xx' \in K^\ast \cap \D_{E'}([1, a]) = \D_K([1, a])\D_K([1, a + b]).\]
    Further using the roundness of quadratic Pfister forms and noting that 
    \[\D_K([1, a + b]) = \im(\norm_{E/K}) \setminus \{0\},\] 
    (recall the notation $E = K(\wp^{-1}(a + b))$) there exists $\gamma \in E^\ast$ such that 
    \[\qpf{\norm_{E / K}(\gamma), a} \cong \qpf{xx', a}.\]

    In $\W_q(K)$, using \cite[Lemma 15.1 (1)]{ElmanKarpenkoMerkurjev2008}, we obtain
    \[\qpf{\norm_{E / K}(\gamma), a} + \qpf{\norm_{E / K}(\gamma), b} = \qpf{\norm_{E / K}(\gamma), a + b}.\]
    But since $\norm_{E / K}(\gamma) \in \D_K([1, a+b])$, the latter form is hyperbolic, yielding 
    \[\qpf{xx', a} \cong \qpf{\norm_{E / K}(\gamma), a} \cong \qpf{\norm_{E / K}(\gamma), b} \cong \qpf{yy', b}.\]

    Using this isometry, in $\W_q(K)$, we obtain 
    \begin{align*}
        \varphi \perp \varphi' &= x[1, a] \perp y[1, b] \perp x'[1, a] \perp y'[1, b] \\
        &= x\qpf{xx', a} \perp y\qpf{yy', b}\\
        &= x \qpf{\norm_{E / K}(\gamma), a} \perp y \qpf{\norm_{E / K}(\gamma), b} \\
        &= \qpf{x, \norm_{E / K}(\gamma), a} \perp \qpf{\norm_{E / K}(\gamma), a} \\
        &\quad\quad\quad \perp \qpf{y, \norm_{E / K}(\gamma), b} \perp \qpf{\norm_{E / K}(\gamma), b} \\
        &= \qpf{\norm_{E / K}(\gamma), x, a} \perp \qpf{\norm_{E / K}(\gamma), y, b}\\
        &= \bipf{\norm_{E / K}(\gamma)} \otimes \varphi.
    \end{align*}

    Recall that by \Cref{eq_TrVarphiInI3}, we have $s_\ast(\varphi), s_\ast(\varphi') \in \GP_3(F)$.
    So using \cite[Proposition 3.3]{BarryChapmanLaghribi_DescentBiquaternion} ($K / F$ separable) resp. \cite[Proposition 5.1 (ii)]{AravireLaghribiORyan_TransferInsep} ($K / F$ inseparable), we obtain
    \begin{align} \label{eq_e3TraceSum}
        e^3(s_\ast(\varphi)) + e^3(s_\ast(\varphi')) &= e^3(s_\ast(\varphi \perp \varphi'))\nonumber \\
        &= s_\ast(e^3(\varphi \perp \varphi')) \nonumber\\
        &= s_\ast(e^3(\bipf{\norm_{E / K}(\gamma)} \otimes \varphi)) \nonumber\\
        &= s_\ast\left({\overline{\dlog\norm_{E / K}(\gamma)}} \wedge [D_K] \right).
    \end{align}
    Using Frobenius Reciprocity and \Cref{lem_basicComputationsTransfer} \ref{lem_basicComputationsTransfer1} together with the compatibility of the norm with field extensions, the latter expression equals
    \[\overline{\dlog \norm_{E / F}(\gamma)} \wedge [D].\]
    Thus, 
    \[e^3(s_\ast(\varphi)) \in H_2^3(F) / (N(d; a, b) \wedge [D])\]
    does not depend on the choice of $x, y \in K^\ast$ in \Cref{eq_D_KAlbertForm} as claimed.
\end{proof}

The following easy case will be used frequently.

\begin{corollary} \label{cor_inv0ifDefOverF}
    If $[D_K] = [\quat{a, x} \otimes \quat{b, y}]$ with $x, y \in F^\ast$ (which implies $\delta_{K / F}(D_K) = 0$ if $K/F$ is separable), then $\inv(D) = 0$.
\end{corollary}
\begin{proof}
    If $x, y \in F^\ast$, then $\varphi = [1, a + b] \perp x[1, a] \perp y[1, b]$ is defined over $F$ and so, $s_\ast(\varphi) = 0$, leading to $\inv(D) = 0$.
\end{proof}

\begin{remark} \label{rem_EntriesOfBiquaternion}
	The argument in \Cref{prop_DK_ExistenceInv} shows, that if $\gamma \in E^\ast$ occuring in \Cref{eq_e3TraceSum} fulfils $\qpf{xx', a} \cong \qpf{\norm_{E/K}(\gamma), a}$, then 
	\[e_3(s_\ast(\varphi)) + e_3(s_\ast(\varphi')) = \overline{\dlog(\norm_{E/F}(\gamma))} \wedge [D].\]
\end{remark}

In the following remark, we compare our invariant to possible generalizations to other triquadratic field extensions.
We will see that the other cases do not admit an analogous invariant or that such an invariant would not carry any information because it vanishes for all algebras in question.

\begin{remark} \label{rem_OtherExtensions}
    So far, we only considered the case in which $L/K$ is a separable multiquadratic extension of degree 4.
    
    If the separability degree of $L / F$ is at most 2 and if $D \in \BrTwo(L/F)$, then $[D_K]$ is given by the class of a biquaternion algebra defined over $F$ by \cite[Corollary 1 (2)]{AravireLaghribiMultiQuadratic}.
    Hence, in these cases, the transfer of the corresponding Albert form will be 0 and hence, $\inv(D) = 0$.
    
    The only case not considered so far is thus  
    \[L = K(\sqrt a, \wp^{-1}(b)).\] 
    with $K / F$ separable. 
    But for the definition of our invariant, we used an intermediate field lying between $K$ and $L$ which is neither $K(\wp^{-1}(a))$ nor $K(\wp^{-1}(b))$, namely $E = K(\wp^{-1}(a + b))$.
    By basic field theory, there is no such field if $L = K(\sqrt a, \wp^{-1}(b))$, which shows that no direct analogue of our invariant exists in this case.
\end{remark}

\section{Triviality of the Invariant} \label{sec_InvTrivial}

The aim of this section is to give conditions under which $\inv(D) = 0$ for a given central simple $F$-algebra $D$ which is split over $L$.
In particular, we will point out connections to the decomposibility of its underlying division algebra, see \Cref{prop_invTrivial}.

\begin{lemma} \label{lem_AlbertDefinedOverF}
    If $\inv(D) = 0$, then there are $x, y \in K^\ast$ such that $[D_K] = [\quat{a, x} \otimes \quat{b, y}$ and the corresponding Albert form
   \[[1, a + b] \perp x [1, a] \perp y[1, b]\]
    is defined over $F$. 
\end{lemma}
\begin{proof}
    Suppose $\inv(D) = 0$.
    We have to prove that $[D_K] = [\quat{a, x} \otimes \quat{b, y}$ for some $x, y \in K^\ast$ such that the corresponding Albert form is defined over $F$.
    By the first part of \Cref{prop_DK_ExistenceInv}, there are $\Tilde x, \Tilde y \in K^\ast$ such that 
    \[[D_K] = [\quat{a, \Tilde x} \otimes \quat{b, \Tilde y}].\]
    We consider the Albert form
    \[\Tilde\varphi = [1, a + b] \perp \Tilde x [1, a] \perp \Tilde y[1, b]\]
    corresponding to $\quat{a, \Tilde x} \otimes \quat{b, \Tilde y}$.
    If $\Tilde\varphi$ is defined over $F$, we are done.
    
    Otherwise, since $\inv(D) = 0$, by definition of $\inv$, there is $\gamma \in E^\ast$ such that 
    \begin{align} \label{eq_ChoiceGamma}
        e^3(s_\ast(\tilde \varphi)) = \overline{\dlog(\norm_{E/F}(\gamma))} \wedge [D]).
    \end{align}
    Let $x = \Tilde x^{-1} \cdot \norm_{E/K}(\gamma)$ and $ y = \Tilde y^{-1} \cdot \norm_{E/K}(\gamma)$.
    Using
    \begin{align}\label{eq_normRepresented}
        \norm_{E/K}(\gamma) \in \D_K([1, a + b]),
    \end{align}
    in $\BrTwo(K)$ we have 
    \begin{align*}
        \quat{a, x} \otimes \quat{b, y} \otimes \quat{a, \Tilde{x}} \otimes \quat{b, \Tilde{y}} &= \quat{a, x\Tilde{x}} \otimes \quat{b, y\Tilde{y}}\\
        &= \quat{a, \norm_{E/F}(\gamma)} \otimes \quat{b, \norm_{E/K}(\gamma)}\\
        &= \quat{a + b, \norm_{E/K}(\gamma)} = 0
    \end{align*}
    and thus 
    \[\quat{a, x} \otimes \quat{b, y} \cong \quat{a, \Tilde{x}} \otimes \quat{b, \Tilde{y}}.\]
    Therefore, 
    \[{\varphi} = [1, a + b] \perp x [1, a] \perp y[1, b]\]
    is also an Albert form of $D_K$.
    Since we have $\quat{a, x\Tilde{x}} \cong \quat{a, \norm_{E/K}(\gamma)}$, we have 
    \[e^3(s_\ast(\varphi)) + e^3(s_\ast(\Tilde{\varphi})) = \overline{\dlog \norm_{E/F}(\gamma)} \wedge [D],\]
    by \Cref{rem_EntriesOfBiquaternion} and by the choice of $\gamma$.
    Using \eqref{eq_ChoiceGamma} we obtain 
    \[e^3(s_\ast({\varphi})) = 0.\]
    Now since $s_\ast({\varphi}) \in \GP_3(F)$, this form is hyperbolic by the Arason-Pfister Hauptsatz, so ${\varphi}$ is defined over $F$.
\end{proof}

Note that in the above proof, when defining $x$ and  $y$, we differed from Sivatski's proof because we considered the norm $\norm_{E/K}$ instead of the norm $\norm_{E/F}$.
But in fact, this does not change the argument since \eqref{eq_normRepresented} holds in both cases.

The next auxiliary result shows that vanishing of $\inv$ does not change when tensoring with biquaternion algebras that have a decomposition corresponding to the fields under consideration.

\begin{lemma} \label{lem_reducing}
    Let $A \in \BrTwo(L / F)$ with $\inv(A) = 0$.
    Then for all $c, e \in F^\ast$, we have 
    \[\inv(A \otimes \quat{a+b, c} \otimes \quat{d, e}) = 0 \quad \text{ resp. } \quad \inv(A \otimes \quat{a+b, c} \otimes \quat{e, d}) = 0,\]
    depending on whether $K / F$ is separable or not.
\end{lemma}
\begin{proof}
    By \Cref{lem_AlbertDefinedOverF}, there are $x, y \in K^\ast$ such that 
    \[[A_{K}] = [\quat{a, x} \otimes \quat{b, y}]\]
    and such that $x [1, a] \perp y[1, b]$ is defined over $F$.
    For all $c \in F^\ast$, we have
    \[[(A \otimes \quat{a + b, c})_K] = [\quat{a, x} \otimes \quat{b, y} \otimes \quat{a, c} \otimes \quat{b, c}] = [\quat{a, x c} \otimes \quat{b, y c}]\]
    and clearly, its Albert form $[1, a+b] \perp x c[1, a] \perp y c [1, b]$ is also defined over $F$.
    
    Obviously, this is also an Albert form for $(A \otimes \quat{a+b, c} \otimes \quat{d, e})_K$ for all $e \in F^\ast$ since $\quat{d, e}_K$ is split if $K/F$ is separable.
    Similarly it is an Albert form for $(A \otimes \quat{a+b, c} \otimes \quat{e, d})_K$ for all $e \in F$ if $K/F$ is inseparable.
    The assertion now follows from \Cref{cor_inv0ifDefOverF}.
\end{proof}

We are now ready to prove the main statement of this section.

\begin{proposition} \label{prop_invTrivial}
    The following conditions are equivalent:
    \begin{enumerate}
        \item\label{prop_invTrivial1} we have $\inv(D) = 0$;
        \item\label{prop_invTrivial2} there is a quaternion algebra $Q$ over $F$, $v \in F^\ast$ and $u \in F^\ast$ resp. $u \in F$ such that 
        \[[D] = \begin{cases}
            [Q \otimes \quat{a + b, v} \otimes \quat{d, u}], & \text{ if $K/F$ is separable}\\
            [Q \otimes \quat{a + b, v} \otimes \quat{u, d}], & \text{ if $K/F$ is inseparable;}
            \end{cases}
        \]
        \item\label{prop_invTrivial3} The division algebra Brauer equivalent to $D_E$ is defined over $F$. 
    \end{enumerate}
\end{proposition}
\begin{proof}
    \ref{prop_invTrivial1} $\Rightarrow$ \ref{prop_invTrivial2}:
        By \Cref{lem_AlbertDefinedOverF} there are $x, y \in K^\ast$ such that 
        \[[D_K] = [\quat{a, x} \otimes \quat{b, y}].\]
        and the Albert form
        \[\varphi = [1, a + b] \perp x [1, a] \perp y[1, b]\]
        for $\quat{a, x} \otimes \quat{b, y}$ is defined over $F$.
        Since $[1, a + b]$ is defined over $F$, this implies $x [1, a] \perp y[1, b]$ to be defined over $F$.
    
        Hence, by comparing Arf invariants, there are $c_1, c_2, c_3 \in F^\ast$ such that in $\W_q(K)$, we have
        \[{\varphi} = [1, a+b] + [c_1, c_2] + c_3[1, c_1c_2 + a + b] = \qpf{c_3, a+b} + c_1 \qpf{c_1c_3, c_1c_2},\]
        yielding $[D_K] = [\quat{a+b, c_3} \otimes \quat{c_1c_2, c_1c_3}]$.
        We thus have 
        \[[D \otimes \quat{a+b, c_3} \otimes \quat{c_1c_2, c_1c_3}] \in \BrTwo(K/F)\]
        which readily implies
        \[[D] = [\quat{c_1c_2, c_1c_3} \otimes \quat{a+b, c_3} \otimes \widehat Q],\]
        where $\widehat Q$ is split over $K$.
        The assertion now follows by the description of $\BrTwo(K/F)$ given in the introduction.

    \ref{prop_invTrivial2} $\Rightarrow$ \ref{prop_invTrivial1}:
        By \Cref{lem_reducing}, it is enough to consider the case in which $D$ is a quaternion algebra. Let $\pi \in \GP_2(F)$ be its associated norm form.
        Let $\tau$ be a form of dimension 4 that is Witt equivalent to $\pi \perp [1, a + b]$.
        If $\tau$ is isotropic, then $[1, a+b]$ is a subform of $\pi$ and thus
        \[\pi \cong \qpf{c, a + b} = \qpf{c, a} + \qpf{c, b} \]
        for some $c \in F^\ast$.
        This case thus follows from \Cref{cor_inv0ifDefOverF}.
        Let now $\tau$ be anisotropic. 
        We first consider the case in which $\tau_K$ is isotropic.
        \begin{itemize}
            \item Let $K / F$ be separable.
            Then by \cite[Chapter V, (4.2) Theorem]{Baeza_SemilocalRings} and by comparing the Arf invariant, we have $\tau \cong x[1, d] \perp y[1, d + a + b]$ for some $x, y \in F^\ast$.
            In $\W_q(F)$ we then obtain 
            \begin{align*}
                \pi = x[1, d] + y[1, d + a + b] + [1, a+b] = x\qpf{xy, d} + \qpf{y, a} + \qpf{y, b}
            \end{align*}
            so in $\W_q(K)$, we have
            \begin{align*}
                \pi_K = \qpf{y, a} + \qpf{y, b}
            \end{align*}
            and we conclude again by \Cref{cor_inv0ifDefOverF}.
            \item Let $K / F$ be inseparable.
            Then by \cite[Corollary 2.4]{Ahmad94} and by comparing the Arf invariant, we have $\tau \cong x[1, s] \perp xd[1, s + a + b]$ for some $x, y, s \in F^\ast$.
            In $\W_q(K)$, we obtain
            \begin{align*}
                \pi_K &= x[1, s] + xd[1, s + a + b] + [1, a+b] \\
                &=x[1, a + b] + [1, a+b] = \qpf{x, a} + \qpf{x, b}
            \end{align*}
            and the conclusion follows as before by \Cref{cor_inv0ifDefOverF}.
        \end{itemize}
        Let now $\tau_K$ be anisotropic.
        Over $L$, $\pi$ and $\tau$ are Witt equivalent.
        Since $D$ splits over $L$, so does $\pi$ resp. $\tau$.
        Hence, if $\tau_{K(\wp^{-1}(a))}$ were anisotropic, it would be isometric to a two-fold Pfister form and thus of Arf invariant 0 in $K(\wp^{-1}(a))$, contradicting $[L : F] = 8$ and $\Delta(\tau) = a + b$.
        Therefore $\tau_{K(\wp^{-1}(a))}$ is isotropic and it follows $\tau_K \cong x[1, a] \perp y [1, b]$ for some $x, y \in K^\ast$.
        
        Now consider the Albert form $\varphi = \tau \perp [1, a + b]$ and its Clifford algebra $\widehat D$.
        By the above, we have $\widehat D_K = \quat{a, x} + \quat{b, y}$, so its Albert form corresponding to this decomposition is given by 
        \[\varphi_K = [1, a+b] \perp x[1, a] \perp y [1, b] = (\tau \perp [1, a + b])_K.\]
        Since this form is defined over $F$, its transfer vanishes and we obtain $\inv(\widehat D) = 0$.
        
        Since we have
        \[[1, a + b] + \tau = \pi\]
        we have $[D] = [\widehat D]$.
        The conclusion now follows by applying \Cref{lem_reducing} again.

    \ref{prop_invTrivial2} $\Rightarrow$ \ref{prop_invTrivial3}:
        Under the given assumption we have $[D_E] = [Q_E]$ and the conclusion is clear.

    \ref{prop_invTrivial3} $\Rightarrow$ \ref{prop_invTrivial2}:
        First note that since $E \subseteq L$ and $[E : F] = 4$, we have $\ind(D_E) \leq 2$ by \cite[Corollary 4.5.11]{GilleSzamuely_CSAGalCohomology}, so $[D_E] = [Q]$ for some quaternion algebra $Q$.
        If $Q$ can be chosen as a quaternion algebra over $F$, we have
            \[[(D \otimes Q)_E] = 0\]
        and the assertion follows by the description of the kernel in the biquadratic case.
\end{proof}

\begin{corollary} \label{cor_Inv0ImpliesDelta0}
	Let $K/F$ be separable.
	If $\inv(D) = 0$, then $\delta_{K/F}(D_K) = 0$.
\end{corollary}
\begin{proof}
	It is clear that by \Cref{prop_invTrivial} that $\inv(D) = 0$ implies that the biquaternion algebra Brauer equivalent to $D_K$ is defined over $F$ and the conclusion follows from \cite[Theorem 3.6]{BarryChapmanLaghribi_DescentBiquaternion}.
\end{proof}

Alternatively, \Cref{cor_Inv0ImpliesDelta0} could have been directly derived from the definitions of the two invariants. 
In fact, if $\inv(D) = 0$ and $\varphi$ is an Albert form corresponding to $D_K$, by definition, there is some $\gamma \in E^\ast$ with $e^3(\varphi) = \overline{\dlog(\norm_{E/F}(\gamma))} \wedge [D]$.
Since $\norm_{E/F}(\gamma) = \norm_{K/F}(\norm_{E/K}(\gamma))$, we conclude $e^3(\varphi) \in \overline{\dlog(\im(\norm_{K/F})\setminus\{0\})} \wedge [D]$, yielding $\delta_{K/F}(D_K) = 0$.

We continue our discussion on analogous invariants for other types of triquadratic extensions.

\begin{remark}
    In \Cref{rem_OtherExtensions}, we pointed out that the analogous invariant for triquadratic extensions of separability degree $\leq 2$ always vanishes. 
    If $F(\sqrt a, \sqrt b) \subseteq L$, this goes along with \Cref{prop_invTrivial} as we will now see.
    
    We know that the Brauer class of any $D \in \BrTwo(L/F)$ is represented by $\quat{x, a} \otimes \quat{y, b} \otimes Q$, where $Q$ is of the form $\quat{d, z}$ or $\quat{z, d}$, depending on whether $K / F$ is separable or not.
    The formulation in \Cref{prop_invTrivial} can thus be adapted as follows:
    
    In \ref{prop_invTrivial2}, we replace $\quat{a + b, v}$ with $\quat{w, ab}$ and in \ref{prop_invTrivial3} we replace $E = K(\wp^{-1}(a + b))$ with $K(\sqrt{ab})$.
\end{remark}

We conclude this section with an example of an algebra $A$ such that $\inv(A) \neq 0$ but $\delta_{K / F}(A) = 0$, in the case where $K / F$ is a separable quadratic extension, so the reverse implication in \Cref{cor_Inv0ImpliesDelta0} is not true in general.

We will make use of the residue maps constructed in \cite[Section 6]{BarryChapmanLaghribi_DescentBiquaternion}. 
For the reader's convenience, we will recall some basic facts about these residue maps.
Let $\mathcal O$ be a discrete valuation ring with quotient field $M$ and residue field $\overline M$.
Let $\pi$ be a uniformizer.
For $n \geq 1$, let $\I^n_q(M)' := \I^{n-1}(M) \otimes \{[1, a] \mid a \in \mathcal O\}$.
Denoting by $\partial_\pi^1, \partial_\pi^2: \W(M) \to \W(\overline M)$ the usual first and second residue homomorphisms for bilinear forms, we have well defined homomorphisms $\delta_\pi, \Delta_\pi: \W_q(M)' \to \W_q(\overline M)$ with
\[\delta_\pi^2(\bb \otimes [1, a]) = \partial_\pi^2(\bb) \otimes [1, \overline a], \quad \quad \Delta_\pi(\bb \otimes [1, a]) = (\partial_\pi^1 +\partial_\pi^2)(\bb) \otimes [1, \overline a],\]
where $\overline a$ denotes the residue class of $a$.
We refer to $\delta_\pi^2$ as the \emph{second residue homomorphism}.
These maps further induce homomorphisms
\[\delta_\pi^2: \I^n_q(M)' / \I^{n+1}_q(M)' \to \I^{n-1}_q(\overline M) / \I^{n}_q(\overline M)\]
and
\[ \Delta_\pi: \I^n_q(M)' / \I^{n+1}_q(M)' \to \I^{n}_q(\overline M) / \I^{n + 1}_q(\overline M)\]
in the obvious way.

To define residue maps in Kato-Milne cohomology, similarly as before, we define $H_2^{n+1}(M)'$ to be the subgroup $\nu_M(n) \wedge H_2^1(\mathcal O)$ of $H_2^{n+1}(M)$, where $\nu_M(n)$ denotes the kernel of the Artin-Schreier map $\wp: \Omega_M^n \to \Omega_M^n / \diff \Omega_M^{m-1}$.
Since Kato's isomorphisms restrict to isomorphisms $g_n: H_2^n(M)' \to \I^n_q(M)' / \I^{n+1}_q(M)'$, we obtain residue maps in cohomology by defining
\[\xi_\pi:= f^{n-1} \circ \delta^2_\pi \circ g_{n}\]
and
\[\chi_\pi := f^{n} \circ \Delta_\pi \circ g_{n}.\]

\begin{example} \label{ex_NontrivialInv}
    Let $A$ be a central simple algebra of degree 8 and exponent 2 which is indecomposable, i.e. not isomorphic to a tensor product of 3 quaternion algebras, see \cite[Section 3]{Rowen_DivisionAlgebraExp2Char2} for the existence of such algebras in characteristic 2.
    By \cite[Theorem 2]{Rowen_DivisionAlgebraExp2Char2}, there are $a_1, \ldots, a_4, b_1, \ldots, b_4 \in F^\ast$ such that
    \[M_2(A) \cong \bigotimes_{i = 1}^4\quat{a_i, b_i}.\]
    Set $L = F(\wp^{-1}(a_1), \wp^{-1}(a_2), \wp^{-1}(a_3))$ and $E = F(\wp^{-1}(a_1), \wp^{-1}(a_2))$.
    Let $D = \quat{a_4, b_4}$ and let $\varphi = b_4[1, a_4] \perp [1, a_3 + a_4]$.
    By \cite[Example 6.1]{LaghribiMukhija_ExcellenceInsepQuartic}, $(\varphi_E)_\an$ is of dimension 2 and not defined over $F$.
    We put
    \begin{align} \label{eq_anisotropicPartPhi}
        (\varphi_E)_\an \cong s[1, a_3]
    \end{align} for some $s \in E^\ast$.
    We consider the biquaternion algebra $\tilde D = D \otimes \quat{a_3, t}$ over the Laurent extension $F\dbrac t$.
    We show that $\inv(\tilde D) \neq 0$, where we take the invariant with respect to the triple $(a_1, a_3, a_2)$ (note the order!).
    
    Suppose otherwise that $\ind(\tilde D) = 0$.
    Then by \Cref{prop_invTrivial}, there are $c_1, c_2, e, f \in F\dbrac t$ such that
    \begin{align} \label{eq_expressionTildeD}
        [\tilde D] = [D \otimes \quat{a_3, t}] = [Q \otimes \quat{a_1 + a_3, c_1} \otimes \quat{a_2, c_2}],
    \end{align}
    where $Q = \quat{e, f}$ with $e, f \in F\dbrac t$ and $f \neq 0$.
    Using the Clifford invariant, \eqref{eq_expressionTildeD} implies
    \begin{align*}
    	b_4[1, a_4] \perp t[1, a_3] \perp [1, a_3 + a_4] &\perp f[1, e] \perp c_1[1, a_1 + a_3] \perp \\
	c_2[1, a_2] &\perp [1, e + a_1 + a_2 + a_3] \in \I_q^3(F\dbrac{t}).
    \end{align*}
    Extending scalars to $E$, plugging in \eqref{eq_anisotropicPartPhi} and rearraging some terms, we see
    \begin{align} \label{eq_representationInI3E}
        \bif{s, t, c_1, stc_1} \otimes [1, a_3] \perp stc_1[1, a_3] \perp f[1, e] \perp [1, e + a_3] \\
        = s[1, a_3] \perp t[1, a_3] \perp c_1[1, a_3] \perp f[1, e] \perp [1, e + a_3] \in \I_q^3(E\dbrac{t}).\nonumber
    \end{align}
    Since the latter form is of dimension 10, it is Witt equivalent to a scalar multiple of a Pfister form \cite[Proposition 82.1]{ElmanKarpenkoMerkurjev2008}.
    Since the form becomes hyperbolic over the function field of $\pi = \bif{s, t, c_1, stc_1} \otimes [1, a_3] \in \GP_3(E\dbrac{t})$ as can be verified using the representation in \eqref{eq_representationInI3E}, we can more precisely conclude that this form is Witt equivalent to some scalar multiple of $\pi$ over $E$ by \cite[Corollary 23.6]{ElmanKarpenkoMerkurjev2008}.
    From \eqref{eq_representationInI3E}, we therefore obtain that $stc_1[1, a_3] \perp f[1, e] \perp [1, e + a_3]$ is hyperbolic over $E\dbrac t$, which by considering its Clifford algebra implies
    \[\quat{e, f} \cong \quat{a_3, stc_1},\]
    hence
    \begin{align*}
    	\qpf{f, e} \cong \qpf{stc_1, a_3}.
	\end{align*}
    We will consider the second residue homomorphism of these forms and will distinguish between the cases in which $c_1$ is a unit with respect to the $t$-adic valuation and in which $c_1$ is not a unit.
    \begin{itemize}
        \item $c_1$ unit:
            In this case, we obtain
            \begin{align*}
                \delta_t^2(\qpf{stc_1, a_3}) = s\overline{c_1}[1, a_3].
            \end{align*}
            We compare the Arf invariants and use that $\qpf{e, f}$ is defined over $F\dbrac t$, implying that its residue form represents an element in $F$, and thus we have 
            \[\delta_t^2(\qpf{f, e}) = p[1, a_3]\]
            for some $p \in F^\ast$.
            This now implies
            \[(\varphi_E)_\an \cong p\overline {c_1}[1, a_3],\]
            which is a contradiction to the fact that $(\varphi_E)_\an$ is not defined over $F$.
        \item $c_1$ not unit: 
            In this case, modulo a square, we may assume $c_1 = tu$ for a unit $u$.
            We apply the residue homomorphism to \eqref{eq_expressionTildeD} and obtain 
            \begin{align*}
                \xi_t(\quat{a_1+ a_3, c_1}) &= \xi_t(\quat{a_1+ a_3, tu}) \\
                &= \xi_t(\quat{a_1+ a_3, u}) + \xi_t(\quat{a_1+ a_3, t})\\
                &= 0 + a_1 + a_3 = a_1 + a_3.
            \end{align*}
            Similarly, we have $\delta := \xi_t(\quat{a_2, c_2}) \in \{0, a_2\}$, depending on whether $c_2$ is a unit or not.
            Using \eqref{eq_expressionTildeD} and applying the second residue map again to both sides, we thus have
            \[a_3 = a_1 + a_3 + \delta,\]
            but this contradicts the fact that $[L : F] = 8$.
    \end{itemize}
    
    So we have finally shown $\inv(\tilde D) \neq 0$ with respect to the triple $(a_1, a_3, a_2)$, while clearly $\delta_{F(\wp^{-1}(a_2)) / F}(\tilde D) = 0$ since $\tilde D$ is defined over $F$.
\end{example}

\section{An application to Chow Groups and indecomposable algebras of degree $8$} \label{sec_ChowGroups}

In \cite{Sivatski_CohoInv}, Sivatski constructed an indecomposable central simple algebra $A$ of degree 8 and exponent 2 over a field $K$ of cohomological dimension 3 such that the torsion of the second Chow group of its Severi-Brauer variety is not trivial.
Recall that by \cite[Theorem A.1]{IzhboldinKarpenko_GenericSplittingFields} and \cite[Proposition 5.1]{Karpenko_Codim2Cycles}, we have
\begin{align}\label{eq_IsoTorsChow}
	\frac{\ker(H^3(K) \to H^3(K(\SevBrauer(A))))}{A \cup F^*/F^{*2}} \cong \torsChow(\SevBrauer(A)) \cong \Z/2\Z.
\end{align}

Recall that the second isomorphism was done by Karpenko in any characteristic. 
For the first isomorphism in the setting of Kato-Milne cohomology, we have an injection
\begin{equation}\label{isocohom}
	\frac{\ker(H_2^3(K) \to H_2^3(K(\SevBrauer(A))))}{A \wedge \dlog F^*} \hookrightarrow \torsChow(\SevBrauer(A)).
\end{equation}
In fact, we consider the cohomology groups $H^{m}(F, \Q_2/\Z_2(m-1))$ and\break $H^{m}(F, \Z/2\Z(m-1))$ as defined in \cite[Appendix A]{GaribaldiMerkurjevSerre_CohInvariants}. 
The first group has the second group as $2$-torsion, and it is a subgroup of $H^{m}(F, \Q/\Z(m-1))$. 
Moreover, $H^{m}(F, \Z/2\Z(m-1))=H^1(F, K_{m-1}(F_{sep})/2K_{m-1}(F_{sep}))$, where $K_n(F_{sep})$ is the $n$-th Milnor group of the separable closure $F_{sep}$ of $F$ \cite[Equation (A.4)]{GaribaldiMerkurjevSerre_CohInvariants}. 
One knows that $H^1(F, K_{m-1}(F_{sep})/2K_{m-1}(F_{sep}))\cong H_2^{m}(F)$ (we combine \cite{Kat1} with \cite[(5.12)]{MR1154404}.

Now the injection (\ref{isocohom}) follows since the Kato-Milne cohomology group $H^m_2(F)$ injects into $H^m(F, \Q/\Z(m-1))$, and we have a result due to Peyre \cite[Theorem 4.1]{PeyreProductsSeveriBrauer} that asserts the following isomorphism in any characteristic
\begin{align*}
	\frac{\ker(H^3(K, \Q/\Z(2)) \to H^3(K(\SevBrauer(A)), \Q/\Z(2))}{A \cup F^*/F^{*2}} \cong \torsChow(\SevBrauer(A)).
\end{align*}

Now we produce an analogue of the contruction done by Sivatski in \cite[Section 4]{Sivatski_CohoInv}. 
As mentioned in \Cref{ex_NontrivialInv}, there exists an indecomposable algebra $D$ of degree $8$ and exponent $2$, a separable triquadratic extension $K=F(\wp^{-1}(a_1), \wp^{-1}(a_2), \wp^{-1}(a_3))$ such that $D_K$ is split but $D$ is not isomorphic to $\otimes_{i = 1}^3 [a_i, x_i)$ for any $x_1, x_2, x_3 \in F^\ast$. 
Moreover, we have $M_2(D) \cong \otimes_{i = 1}^4 [a_i, b_i)$ for suitable $a_4, b_1,b_2, b_3, b_4\in F^\ast$.

Let $E = F \dbrac{t_1} \dbrac{t_2} \dbrac{t_3}$ be the field of iterated Laurent series in the independent variables $t_1, t_2, t_3$. 
Let $C = D + [a_1, t_1) + [a_2, t_2) + [a_3,t_3) \in \BrTwo(E)$, and $A$ the division algebra Brauer equivalent to $C$. 
We have the following:
\begin{lemma} \label{lem_AindexIndec} 
	The algebra $A$ is of index $8$ and indecomposable. 
\end{lemma}
\begin{proof}
	Since $A$ is split over $E \cdot K$, it follows that $A$ has index at most $8$. 
	Moreover, extending scalars to $L:=E(\sqrt{t_1}, \sqrt{t_2}, \sqrt{t_3})$ yields $A_L \sim D_L$. 
	Hence, $A$ is of index $8$ and indecomposable, since otherwise, $D$ would be decomposable over $K$ which is not possible by \cite[Lemma 7]{MorandiSethuran_IndecDivisionAlgBaerOrdering}.
\end{proof}

Now since $D_K$ is split, there exist $\alpha, \beta \in M := F(\wp^{-1}(a_3))$ such that $D_M = [a_1,\alpha) + [a_2, \beta)$. 
Hence, $C_{E(\wp^{-1}(a_3))} = [a_1, t_1\alpha) + [a_2, t_2\beta)$. 
Let $\phi = t_1 \alpha [1, a_1] \perp t_2 \beta [1, a_2] \perp [1, a_1+a_2]$ be the Albert form associated to $C_{E(\wp^{-1}(a_3))}$. 
Similarly as in \eqref{eq_TrVarphiInI3}, we have 
\[e^3(\trace_\ast(\phi))=a_1\dlog(yvt_1t_2)\wedge \dlog(\norm_{M/F}(\alpha)),\]
where $\alpha = x + y \wp^{-1}(a_3)$ and $\beta = u + v \wp^{-1}(a_3)$ and the trace with respect to the field extension $E(\wp^{-1}(a_3)) / E$.

Let $E(\SevBrauer(C))$ be the function field of the Severi-Brauer variety of $A$. 
As was shown by Sivatski in \cite[Proposition 4.1]{Sivatski_CohoInv}, we have the following result.

\begin{proposition}\label{PropChow}
	The symbol $a_1\dlog(yvt_1t_2) \wedge \dlog(\norm_{M/F}(\alpha))$ is zero over $E(\SevBrauer(C))$, but $a_1\dlog(yvt_1t_2) \wedge \dlog(\norm_{M/F}(\alpha)) \neq [C] \wedge \dlog(f)$ for any $f\in E^\ast$.
\end{proposition}

\begin{proof}
	We follow the proof done by Sivatski doing some adaptations to our setting since some arguments used by Sivatski do not apply in 	characteristic $2$. 
	First, the symbol $a_1 \dlog(yvt_1t_2) \wedge \dlog(\norm_{M/F}(\alpha))$ vanishes over $E(\SevBrauer(C))$ since the Albert form $\phi$ is hyperbolic over $E(\SevBrauer(C))$.

	Now suppose there  exists $f\in E^\ast$ such that
	\begin{align}\label{eqappl1}
		a_1\dlog(yvt_1t_2)\wedge \dlog(\norm_{M/F}(\alpha))= [C]\wedge \dlog(f).
	\end{align}
	Without loss of generality, we may suppose $f\in F \dbrack{t_1}\dbrack{t_2}\dbrack{t_3}$ square free. 
	As was done by Sivatski, $f$ is a unit with respect to the $t_3$-adic valuation of $E$. 
	Now based on the Techm\"uller lifting with respect to the $t_3$-adic valuation of $E$ combined with the results \cite[Theorem A.19, Corollary A.21]{LaghribiMaitiRationalFunctionFields}, we may identify $f$ with its residue class, and thus we may suppose that $f\in F[[t_1]][[t_2]]$. 
	Applying the residue map $\delta_{t_3}^2$ to equation \eqref{eqappl1} yields
	\[[[a_3, f)]=\delta_{t_3}^2([C]\wedge \dlog(f))=\delta_{t_3}^2(a_1\dlog(yvt_1t_2)\wedge \dlog(\norm_{M/F}(\alpha)))=0.\]
	Consequently, $f\in \im(\norm_{F(\wp^{-1}(a_3))/F})$. 
	Taking the $\delta$-invariant, we get
	\[\delta_{E(\wp^{-1}(a_3))/E}(C)=\overline{a_1\dlog(yvt_1t_2)\wedge \dlog(\norm_{M/F}(\alpha))}=\overline{[C]\wedge \dlog(f)}=0.\]
	It follows from \cite[Theorem 3.6]{BarryChapmanLaghribi_DescentBiquaternion} that $C$ is defined over $E$, and thus the algebra $A$ is decomposable, a contradiction.
\end{proof}

Now we are able to prove the analogue of \cite[Theorem 4.2]{Sivatski_CohoInv}.
\begin{theorem}
	Let $a_1, a_2, a_3\in F$ be such that $[F(\wp^{-1}(a_1), \wp^{-1}(a_2), \wp^{-1}(a_3)):F]=8$ and $D$ as fixed at the beginning of this section. 
	Then, there exists an extension $K/F$ and a central simple $K$-algebra $B$ such that the following conditions are satisfied:
	\begin{enumerate}[label=(\arabic*)]
		\item ${\rm cd}_2(K)=3$.
		\item $K$ has no odd degree extension.
		\item $B$ has index $8$ and splits over $K(\wp^{-1}(a_1), \wp^{-1}(a_2), \wp^{-1}(a_3))$.
		\item $\torsChow(\SevBrauer(B)) \cong \Z/2\Z$. 
			In particular, $B$ is indecomposable.
		\item The nonzero element of $\torsChow(\SevBrauer(B))$ is given by a symbol in $H_2^3(K)$ that belongs to the intersection
		\[\left(a_1\dlog(\norm_F(a_3))\wedge \dlog(K)\right) \cap \left(a_2\dlog(\norm_F(a_3))\wedge \dlog(K)\right),\]
		where $\norm_F(a_3) = \im(\norm_{F(\wp^{-1}(a_3)}) \setminus \{0\}$.
	\end{enumerate}
\end{theorem}

\begin{proof}
	We have seen in \Cref{PropChow} that the symbol $\pi:=a_1\dlog(yvt_1t_2)\wedge \dlog(\norm_{M/F}(\alpha))$ satisfies 
	\[\displaystyle{0\neq  \overline{\pi}\in \frac{{\rm Ker}(H^3_2(E)\longrightarrow H^3_2({E(\SevBrauer(C))}))}{[C]\wedge \dlog(E^*)}},\]
	where $C =  D + [a_1, t_1) + [a_2, t_2) + [a_3,t_3) \in \BrTwo(E)$ as above.
	Hence, by \eqref{eq_IsoTorsChow} $\torsChow(\SevBrauer(A))\neq 0$, where $A$ is the division algebra Brauer equivalent to $C$. 
	This algebra $A$ is of index $8$ and exponent $2$ by \Cref{lem_AindexIndec}. 
	Moreover, $\pi$ lies in the intersection because $[a_1, \norm_{F(\wp^{-1}(a_3))/F}(\alpha)) = [a_2, \norm_{F(\wp^{-1}(a_3))/F}(\alpha))$.

	By \cite[Theorem 4.1]{BarryChapmanLaghribi_DescentBiquaternion} there exists an extension $K/E$ such that ${\rm cd}_2(K)=3$ and $\torsChow(\SevBrauer(A_K))\neq 0$. 
	In particular, $\torsChow(\SevBrauer(A_K))\cong \Z/2\Z$ and $A_K$ is indecomposable by \cite[Proposition 5.3]{Karpenko_Codim2Cycles}.
	Moreover, it is shown that $K$ has no odd degree extension since it is $2$-special. 
	Obviously, $A_K$ is split over $K(\wp^{-1}(a_1), \wp^{-1}(a_2), \wp^{-1}(a_3))$.
\end{proof}

%
%

\section{Descent Results} \label{sec_Descent}

\subsection*{Descent for Algebras} 

After having constructed the invariant $\inv$ and discussed how it can be used to detect decomposibility of central simple algebras of exponent 2 that split over a triquadratic extension, we will now turn to the descent problem mentioned in the introduction.
More precisely we will cover how the descent can be realized along a field extension of odd degree in certain cases. 
In fact, we will cover the cases of descending quaternion algebras along mixed biquadratic extensions and descending biquaternion algebras along arbitrary quadratic extensions.

\begin{theorem} \label{thm_DescentAlgebras}
	Let $b \in F\setminus \wp(F)$ and $d \in F\setminus F^2$.
    Let $D \in \BrTwo(F)$ and $M/F$ be an odd degree extension such that $[D_{M(\wp^{-1}(b), \sqrt d)}] = [Q_{M(\wp^{-1}(b), \sqrt d)}]$ for some quaternion algebra $Q$ over $M$.
    Then, there is a quaternion algebra $\tilde Q$ over $F$ such that 
    \[[D_{F(\wp^{-1}(b), \sqrt d)}] = [\tilde Q_{F(\wp^{-1}(b), \sqrt d)}].\]
\end{theorem}
\begin{proof}
    Let $K = F(\sqrt d)$.
    Recall that the index of a central simple algebra does not change when extending scalars to a field extension of degree prime to the index, \cite[Corollary 4.5.11]{GilleSzamuely_CSAGalCohomology}.
    We thus have 
    \[\ind(D_{F(\wp^{-1}(b), \sqrt d)}) = \ind(D_{M(\wp^{-1}(b), \sqrt d)}) \leq 2.\]
    
    Consequently, $\ind(D_K) \leq 4$.
    If $D_{K}$ already has index $\leq 2$. 
    Then $\ind(D) \leq 4$.
    \begin{itemize}
    	\item If $\ind(D) \leq 2$, then $D$ is Brauer equivalent to some quaternion $\tilde Q$, and we are done.
	\item If $\ind(D) = 4$, let $\varphi$ be an Albert form associated to $D$.
	This form is anisotropic and becomes isotropic over $F(\sqrt d)$.
	By \cite[Corollary 2.4]{Ahmad94}, we thus have that $\varphi$ is similar to $x[1, y] \perp xd[1, z] \perp [1, y + z]$ for some $x, y, z \in F^\ast$ and we have $[D] = [\quat{y + z, x} + \quat{z, d}]$. 
	In this case, we take $\tilde Q = \quat{y+z, x}$.
    \end{itemize}

    In the remaining case, $D_K$ has index 4.
    Similarly as before, we obtain $[D_K] = [\quat{b, \beta} \otimes \widehat Q]$ for some $\beta \in K^\ast$ and a quaternion algebra $\widehat Q$ over $K$ (using the isotropy of an Albert form over the separable quadratic extension $K(\wp^{-1}(b)) / K$).
    
    Write $\widehat Q = \quat{a, \alpha}$ for suitable $a, \alpha \in K^\ast$.
    Since $\quat{a, \alpha} \cong \quat{a^2, \alpha}$, we may assume $a \in F$.
    Set further $b_1 = a + b \in F$.
    We have 
    \[[D_{L(\wp^{-1}(b), \sqrt d)}] = [Q_{L(\wp^{-1}(b), \sqrt d)}],\]
    which implies 
    \begin{align} \label{eq_RepresentationD_L}
        [D_L] = [Q_L + \quat{c_1, d} + \quat{a + b_1, c_2}]
    \end{align} 
    for some $c_1, c_2 \in L$ (using the kernel of biquadratic extensions in the setting of Brauer groups).
    Because $\widehat Q$ splits over $K(\wp^{-1}(a))$, we see that $D$ splits over $F(\wp^{-1}(a), \wp^{-1}(b_1), \sqrt d)$ and thus, $D_L$ splits over $L(\wp^{-1}(a), \wp^{-1}(b_1), \sqrt d)$.
    So we may apply $\inv$ to $D_L$ resp. $D$ with respect to the triple $(a, b_1, d)$.
    Now \eqref{eq_RepresentationD_L} together with \Cref{prop_invTrivial} implies $\inv(D_L) = 0$.
    
    Let $\inv(D) = \overline u$ for some $u \in H_2^3(F)$.
    Since $\inv(D_L) = 0$, there is some $\gamma \in L(\wp^{-1}(a + b_1), \sqrt d)^\ast$ such that 
    \[\overline u_L = \overline{\dlog\norm_{L(\wp^{-1}(a + b_1), \sqrt d) / L}(\gamma)} \wedge [D_L].\]
    Because $[L:F]$ is odd and $\cha(F) = 2$, we get that $L/F$ is separable.
    Let $s: L \to F$ be the usual trace map and $s_\ast$ its induced transfer on cohomology.
    Using Frobenius reciprocity, we obtain
    \begin{align*}
        u = [L:F]\cdot u &= \overline{\dlog \norm_{L / F}(\norm_{L(\wp^{-1}(a + b_1), \sqrt d) / L}(\gamma))} \wedge [D]\\
        &= \overline{\dlog \norm_{F(\wp^{-1}(b), \sqrt{d}) / F}(\norm_{L(\wp^{-1}(a + b_1), \sqrt d) / F(\wp^{-1}(b), \sqrt{d})}(\gamma))} \wedge [D],
    \end{align*}
    implying $\inv(D) = 0$ (again with respect to the triple $(a, b_1, d)$).
    
    Using \Cref{prop_invTrivial} again, we have 
    \[[D] = [\Tilde{Q} + \quat{a + b_1, x} + \quat{y, d}]\]
    for some $F$-quaternion $\tilde Q$ and scalars $0\neq x, y \in F$.
    Since $a + b_1 = b$ by construction, we have
    $[D_{F(\wp^{-1}(b), \sqrt d)}] = [\Tilde{Q}_{F(\wp^{-1}(b), \sqrt d)}]$ as asserted.
\end{proof}

\begin{remark}
	In \Cref{thm_DescentAlgebras}, we restricted ourselves to the case of mixed biquadratic extensions. 
	For the separable case, one can show that $\widehat Q$ is either of the form $\quat{a, \alpha}$ or $\quat{\alpha, a}$, where $\alpha \in K^\ast$ and $a \in F$.
	In the first case, we may argue as above, but for the second case it does not seem possible to combine $\widehat Q$ and $\quat{b, \beta}$ in a meaningful way.
	This goes along with the fact that there are no intermediate fields between $K = F(\wp^{-1}(d))$ and $L = K(\sqrt a, \wp^{-1}(b))$ other than $K(\sqrt a)$ and $K(\wp^{-1}(b))$, as we already pointed out in \Cref{rem_OtherExtensions}.
\end{remark}

We include a short discussion on the invariant $\delta_{K/F}$ introduced in characteristic 2 in \cite{BarryChapmanLaghribi_DescentBiquaternion} and how the results can be transferred to the case in which $K = F(\sqrt d)$ is an inseprable quadratic extension.

Instead of relying on the trace as in \cite{BarryChapmanLaghribi_DescentBiquaternion} (which vanishes and thus does not carry any information), we use the linear map 
\begin{align} \label{eq_LinearMapInsep}
	s: K \to F\quad \text{with} \quad s(1) = 0 \text{ and }s(\sqrt d) = 1
\end{align} 
and its induced transfer $s_\ast: \W_q(K) \to \W_q(F)$, in line with the construction of the invariant $\inv$.

With this notation, \cite[Proposition 3.3]{BarryChapmanLaghribi_DescentBiquaternion} was transferred to the inseparable case in \cite[Proposition 5.1 (ii)]{AravireLaghribiORyan_TransferInsep}.
We can thus extend \cite[Proposition 3.4]{BarryChapmanLaghribi_DescentBiquaternion} and have the following:
\begin{proposition} \label{prop_TransferVanishesCondition}
	For any $\lambda \in K^\ast$ and any Albert form $\varphi$ over $F$ with Clifford invariant $B$, we have $s_\ast(\lambda\varphi) = 0$ if and only if $e^3(s_\ast(\varphi)) = s_\ast(\overline{\dlog(\lambda)} \wedge[B])$.
\end{proposition}

Building on this we can define the alternative version $\delta_{K/F}(B)$ for an inseparable quadratic extension $K/F$ and a biquaternion algebra $B$ with Albert form $\varphi$ and $s_\ast([B])$ trivial as the class of $e^3(s_\ast(\varphi))$ in the quotient
\[H_2^3(F) / s_\ast(\overline{\dlog K^\ast} \wedge [B]).\]

The following result holds for the invariant $\delta_{K / F}$ with the similar proof as in the separable case:

\begin{proposition} \label{prop_BiquaternionDescentIffDelta0}
	Let $K / F$ be an inseparable quadratic extension, $K = F(\sqrt d)$ and $s : K \to F$ be the map from \eqref{eq_LinearMapInsep} and $s_\ast : \BrTwo(K) \to \BrTwo(F)$ its induced transfer.
	Let $B$ be a biquaternion algebra over $K$ with $s_\ast([B]) = 0$.
	Then $B$ has a descent to $F$ if and only if $\delta_{K / F}(B) = 0$.
\end{proposition}

After having discussed the generalization of $\delta_{K / F}$ for inseparable extensions, we turn our attention back to descents of algebras with respect to odd degree extensions.
Now we will cover biquaternion algebras.
Note that the proof of the following result does not rely on $\inv$ but on $\delta_{K / F}$.

\begin{proposition} \label{prop_BiquaternionDescent}
    Let $L/F$ be an odd degree field extension, $K = F(\theta)$ a quadratic extension (separable or inseparable) of $F$ and $D \in \BrTwo(F)$.
    Suppose there is a biquaternion algebra $B$ over $L$ such that $[D_{L(\theta)}] = [B_{L(\theta)}]$.
    Then there is a biquaternion algebra $\Tilde{B}$ over $F$ such that $[D_{K}] = [\tilde B_{K}]$. 
\end{proposition}
\begin{proof}
    As in the proof of \Cref{thm_DescentAlgebras}, we see that $\ind(D_K) \leq 4$, so we can choose an Albert form $\varphi$ corresponding to $D_K$.
    By assumption, we have $\delta_{L(\theta) / L}(D_{L(\theta)}) = 0$, because the division algebra Brauer equivalent to $D_{L(\theta)}$ is defined over $L$.
    By \cite[Proposition 3.4 (3), Theorem 3.6]{BarryChapmanLaghribi_DescentBiquaternion} (resp. its inseparable analogues \Cref{prop_TransferVanishesCondition} and \Cref{prop_BiquaternionDescentIffDelta0}), there is $\lambda \in L(\delta)^\ast$ such that $s_{L(\theta)/L, \ast}(\lambda \varphi) = 0$, where again $s_{L(\theta)/L, \ast}$ is the transfer homomorphism induced by the $L$-linear map $s_{L(\theta)/L}: L(\theta) \to L$ with $s_{L(\theta)/L}(1) = 0$ and $s_{L(\theta)/L}(\theta) = 1$.
    
    Since $L / F$ is separable (because it is of odd degree), we find an element $\mu \in L$ such that $L = F(\mu)$.
    Let $t_{L/F}: L \to F$ be the $F$-linear map with $t_{L/F}(1) = 1$ and $t_{L/F}(\mu^k) = 0$ for $k \in \{1, \ldots, n\}$.
    Similarly, we can define $s_{K/F}$ and $t_{L(\theta)/ K}$.
    By checking on an $F$-basis of $L(\theta)$ we readily verify $s_{K/F} \circ t_{L(\theta) / K} = t_{L/F} \circ s_{L(\theta) / L}$.
    
    By considering the induced transfers, we thus have a commutative diagram
    \begin{align}\label{eq_commDiagramTrace}
        \begin{diagram}
            & & \W_q(L(\theta)) & & \\
            & \ldTo^{s_{L(\theta) / L, \ast}} & & \rdTo^{t_{L(\theta)/K, \ast}}\\
            \W_q(L) & & & & \W_q(K)\\
            & \rdTo_{t_{L/F, \ast}} & & \ldTo_{s_{K / F, \ast}} \\
            & & \W_q(F) & &
        \end{diagram}
    \end{align}
    Let $\Phi$ be an Albert form of $D_L$.
    Since Albert forms for the same biquaternion algebra are similar, there is $\alpha \in L(\theta)^\ast$ such that $\Phi = \alpha \varphi_{L(\theta)}$.
    
    Since $\Phi$ is defined over $L$, when taking the left path in \eqref{eq_commDiagramTrace}, we see that $\Phi_{L(\theta)} = \alpha\varphi_{L(\theta)}$ is mapped to zero in $\W_q(L)$ and hence to zero in $\W_q(F)$.
    On the other hand when taking the right path, we first see
    \[t_{L(\theta)/K, \ast}(\Phi_{L(\theta)}) = t_{L(\theta)/K, \ast}(\alpha\varphi_{L(\theta)}) = t_{L(\theta)/K, \ast}(\bif{\alpha}) \otimes \varphi.\]
    Here $t_{L(\theta)/K, \ast}(\bif{\alpha})$ is a bilinear form of odd dimension $[L(\theta) : K] = [L : F]$, so it is Witt equivalent to $\bif{x} \perp \psi$, where $\psi \in \I^2(K)$.
    Applying $s_{K/F, \ast}$ and comparing with the result from the left path in \eqref{eq_commDiagramTrace} now gives us 
    \[0 = s_{K/F, \ast}(x\varphi) \perp s_{K/F, \ast}(\psi \otimes \varphi).\]
    Using \cite[Corollary 21.5]{ElmanKarpenkoMerkurjev2008}, we see $s_{K/F, \ast}(\psi \otimes \varphi) \in \I_q^4(F)$, so we also obtain $s_{K/F, \ast}(x\varphi) \in \I_q^4(F)$.
    The Arason-Pfister Hauptsatz implies $s_{K/F, \ast}(x\varphi)$ to be hyperbolic.
    This shows that $x\varphi$ is defined over $F$ by \cite[Theorem 34.9]{ElmanKarpenkoMerkurjev2008} resp. \Cref{lem_TransferExactSequenceInsep}, and the claim follows.
\end{proof}

Note that the commutativity of the analogous diagram for \eqref{eq_commDiagramTrace} in Kato-Milne cohomology is known by \cite[Theorem 2.5]{AravireLaghribiORyan_TransferInsep}.

\subsection*{Odd Degree Descent for Quadratic Forms} 

A field extension $K / F$ is called \emph{excellent} if for any quadratic form over $F$, the anisotropic part $(\varphi_K)_\an$ is defined over $F$, i.e. there is a quadratic form $\psi$ over $F$ such that $(\varphi_K)_\an \cong \psi_K$. 

It is known that quadratic separable extensions and purely inseparable multiquadratic field extensions are excellent, see \cite[Corollary 22.12]{ElmanKarpenkoMerkurjev2008} resp. \cite[Theorem 3.4]{Hof4}.
On the other hand, biquadratic extensions are not excellent in general, see \cite[Example 6.1]{LaghribiMukhija_ExcellenceInsepQuartic} and \cite[Example 3.3]{BingolChapmanLaghribi_MixedSplittingFields} for examples of biquadratic separable resp. biquadratic mixed field extensions that are not excellent.
For a biquadratic field extension $K / F$ which is not purely inseparable and a given quadratic form $\varphi$ over $F$, it is thus desirable to have criteria predicting whether $(\varphi_K)_\an$ is defined over $F$.

Using our invariant $\inv$, especially the descent result for central simple algebras proved in \Cref{thm_DescentAlgebras}, we develop a criterion relative to an odd degree extension for the case of a mixed biquadratic field extension.
More precisely, we have the following:

\begin{theorem} \label{thm_descentQF}
    Let $L/F$ be a finite field extension of odd degree and $\varphi$ be a quadratic form.
    Let $E / F$ be a biquadratic extension of degree 4 such that $(\varphi_{LE})_\an$ is defined over $L$.
    Then $(\varphi_E)_\an$ is defined over $F$ in the following situations:
    \begin{enumerate}
        \item \label{thm_descentQF4} $E = F(\wp^{-1}(b), \sqrt d)$, $\varphi$ is nonsingular, $\dim(\varphi_E)_\an = 4$ and $\Delta(\varphi_E) = 0$;
        \item \label{thm_descentQF3} $E = F(\wp^{-1}(b), \sqrt d)$ and $\dim(\varphi_{E})_\an = 3$;
        \item \label{thm_descentQFleq2} $\dim(\varphi_{E})_\an \leq 2$.
    \end{enumerate}
\end{theorem}
\begin{proof}
    \ref{thm_descentQF4}:
    Since $\Delta(\varphi_E) = 0$, we already have $\Delta(\varphi_{F(\wp^{-1}(b))}) = 0$, so $\Delta(\varphi) \in \{0, b\}$.
    Since the anisotropic part of $\varphi_E$ is the same as the anisotropic part of $(\varphi \perp [1, b])_E$, we may assume $\Delta(\varphi) = 0$, i.e. $\varphi \in \I^2_q(F)$.
    
    Let $[D]$ denote its Clifford invariant.
    Since $(\varphi_{L(\wp^{-1}(b), \sqrt d)})_\an$ is defined over $L$ and of dimension $4$ by assumption, its Clifford invariant $[D_{L(\wp^{-1}(b), \sqrt d)}]$ is equal to $[Q_{L(\wp^{-1}(b), \sqrt d)}]$, where $Q$ is a quaternion algebra defined over $L$.
    By \Cref{thm_DescentAlgebras}, this property descents to $F$, i.e. there is a quaternion algebra $\widehat Q$ over $F$ such that $[D_{F(\wp^{-1}(b), \sqrt d)}] = [\widehat Q_{F(\wp^{-1}(b), \sqrt d)}]$.
    
    We now deduce $Q_{L(\wp^{-1}(b), \sqrt d)} \cong \widehat Q_{L(F(\wp^{-1}(b), \sqrt d)}$.
    Since the norm forms of quaternion algebras are unique up to isometry, we find $\lambda \in L(\wp^{-1}(b), \sqrt d)$ such that $(\varphi_{L(\wp^{-1}(b), \sqrt d)})_\an \cong \lambda\widehat \pi$, where $\widehat\pi$ is the norm form of $\widehat Q$.
    Using the roundness of Pfister forms we can suppose $\lambda \in L$ because $(\varphi_{LE})_\an$ is defined over $L$.
    
    Using \cite[Theorem 2.1 (ii)]{Ahmad04}, we obtain 
    \[\varphi_L = \lambda\widehat\pi + \psi_1\otimes [1, b] + \bipf{d} \otimes \psi_2 \in \W_q(L)\]
    for a bilinear form $\psi \in \W(L)$ and a quadratic form $\psi_2 \in \W_q(L)$.
    
    Since $L / F$ is separable, there is some $\gamma \in L$ such that $L = F(\gamma)$.     
    Let $s: L \to F$ be the linear map with $s(1) = 1$ and $s(\gamma^k) = 0$ for $1 \leq k < [L : F]$.
    Using Frobenius reciprocity, we have
    \begin{align*}
        \varphi &= s_\ast(\lambda\widehat\pi + \psi_1\otimes [1, b] + \bipf{d} \otimes \psi_2)\\
        &= s_\ast(\bif{\lambda})\widehat\pi + s_\ast(\psi_1)\otimes [1, b] + \bipf{d}\otimes s_\ast(\psi_2).
    \end{align*}
    It is easy to see that $s_\ast(\bif{\lambda})$ is Witt equivalent to $\bif{\norm_{F(\lambda) / F}(\lambda)}$.
    The conclusion now follows since $(\varphi_{E})_\an \cong s_\ast(\bif{\lambda}) \otimes \widehat\pi$.
    
    \ref{thm_descentQF3}: 
    Since by \cite[Lemma 2.2]{HL04}, for any quadratic form $\varphi$ (of arbitrary type) over $F$ and any field extension $K/F$, $(\ql{\varphi_K})_\an$ is defined over $F$, it is enough to consider the case in which $(\varphi_E)_\an$ is of type $(1, 1)$.
    More precisely we can write $\varphi = \varphi_r \perp \ql \varphi$ with $\varphi_r$ nonsingular such that $\ql{\varphi_{LE}} \cong \qf{s, 0, \ldots, 0}$, where $s \in F^\ast$.
    Using \cite[Lemma 2.6]{HL04} and the assumption, we get $((\varphi_r \perp \qf{s})_{LE})_\an \cong x[1, a] \perp \qf s$ for some $x, a \in L^\ast$.
    
    We choose the nonsingular completion with trivial Arf invariant.
    Then we get
    \[((\varphi_r \perp s[1, \Delta(\varphi_r)])_{LE})_\an \cong s\qpf{xs, a}.\]
    In particular, we can apply \ref{thm_descentQF4} to $\widehat \varphi:= \varphi_r \perp s[1, \Delta(\varphi_r)]$ and deduce the existence of $\alpha, \beta \in F$ such that
    \[(\widehat \varphi_E)_\an \cong s\qpf{\alpha, \beta}.\]
    Adding $\qf s$ on both sides and using the isometries $(\ql\varphi_E)_\an \cong \qf s_E$ and $s[1, t] \perp \qf s \cong \H \perp \qf s$, we obtain the Witt equivalence
    \begin{align*}
        \varphi_E&\sim (\varphi_r \perp \ql\varphi)_E \\
        &\sim (\varphi_r \perp s[1, \Delta(\varphi_r)]  \perp \qf s \perp \qf{0, \ldots, 0})_E\\
        &\sim (\widehat \varphi_E)_\an \perp \qf s_E \perp \qf{0, \ldots, 0} \\
        &\cong s\qpf{\alpha, \beta} \perp \qf s_E \perp \qf{0, \ldots, 0} \sim \qf s_E \perp s\alpha[1, \beta]_E \perp \qf{0, \ldots, 0}
    \end{align*}
    which yields the assertion by taking the anisotropic part.
    
    \ref{thm_descentQFleq2}: 
    Similarly as in \ref{thm_descentQF3}, it is enough to consider the case in which $(\varphi_E)_\an$ is of type $(1, 0)$.
    Let $a = \Delta(\varphi)$.
    By assumption, there is $x \in L^\ast$ such that $(\varphi_{LE})_\an = x[1, a]$.
    In $\W_q(L)$ we thus have
    \[\varphi_{L} = x[1, a] + \psi_1 + \psi_2,\]
    where, depending on the separability degree of $E / F$, the forms $\psi_i$ are divisible either by $\bipf{a_i}$ or $\qpf{a_i}$ for a suitable $a_i \in L^\ast$, see \cite[Theorems 2.1, 2.2]{Ahmad04}.
    
    Now, we can conclude as in \ref{thm_descentQF4}. 
\end{proof}

\section*{Acknowledgments}

The main part of the research was conducted during two stays of the second author at Universit\'e d'Artois in November 2024 and October 2025.
He expresses his gratitude for the hospitality during his visits.

Furthermore, the authors would like to thank Nikita Karpenko for his comments that the paper of Peyre \cite[Theorem 4.1]{PeyreProductsSeveriBrauer} remains true in arbitrary characteristic.


\bibliographystyle{alpha}
\bibliography{bibliography}

\end{document}